\documentclass[12pt]{amsart}
\usepackage{dsfont}
\usepackage{amssymb, amsmath, amsthm}
\usepackage{mathrsfs}
\newcommand{\no}[1]{#1}
\renewcommand{\no}[1]{}
\no{\usepackage{times}\usepackage[subscriptcorrection, slantedGreek, nofontinfo]{mtpro}
\renewcommand{\Delta}{\upDelta}}
\usepackage{color}
\usepackage{graphicx}
\usepackage{enumerate,comment}

\newtheorem{theorem}{Theorem}[section]
\newtheorem{prop}{Proposition}[section]
\newtheorem{lem}{Lemma}[section]

\newtheorem{cor}{Corollary}[section]

\theoremstyle{remark}

\newcommand{\bel}{\begin{equation} \label}
\newcommand{\ee}{\end{equation}}

\newcommand{\R}{{\mathbb R}}
\newcommand{\N}{{\mathbb N}}

\def\phi {\varphi}

\renewcommand{\leq}{\leqslant}
\renewcommand{\geq}{\geqslant}

\def\beq{\begin{equation}}
\def\eeq{\end{equation}}
\newcommand{\bea}{\begin{eqnarray}}
\newcommand{\eea}{\end{eqnarray}}
\newcommand{\beas}{\begin{eqnarray*}}
\newcommand{\eeas}{\end{eqnarray*}}

\providecommand{\abs}[1]{\left\lvert#1\right\rvert}
\providecommand{\norm}[1]{\left\lVert#1\right\rVert}

\numberwithin{equation}{section}

\title[Determination of quasilinear terms]{Determination of quasilinear terms  from restricted data and point measurements}
\begin{document}
\begin{abstract} We study the inverse problem of determining uniquely and stably quasilinear terms appearing in an elliptic equation from boundary excitations and measurements associated with the solutions of the corresponding equation. More precisely, we consider the determination of quasilinear terms depending simultaneously on the solution and the gradient of the solution of the elliptic equation from  measurements of the flux restricted to some fixed and finite number of points located at the boundary of the domain generated by Dirichlet data lying on a finite dimensional space. Our Dirichlet data will be explicitly given by affine functions taking values in $\R$. We prove our results by considering a new approach based on explicit asymptotic properties of solutions of these class of nonlinear elliptic equations with respect to a small parameter imposed at the boundary of the domain.


\end{abstract}


\author[Yavar Kian]{Yavar Kian}
\address{Univ Rouen Normandie, CNRS, Normandie Univ, LMRS UMR 6085, F-76000 Rouen, France.}
\email{	yavar.kian@univ-rouen.fr}

\maketitle


\section{Introduction}
Let $\Omega$ be a bounded domain of $\mathbb{R}^n$, $n\geq 2$, with $ C^{2+\alpha}$, $\alpha\in(0,1)$, boundary.  Let   $a:=(a_{i,j})_{1 \leq i,j \leq n} \in C^3(\R\times\R^n;\R^{n\times n})$
be symmetric, that is 
\bel{ell1}
a_{i,j}(\mu,\eta)=a_{j,i}(\mu,\eta),\quad  (\mu,\eta)\in\R\times\R^n,\ i,j = 1,\ldots,n, 
\ee
and assume that there exists $\kappa\in C(\R\times\R^n;(0,+\infty))$ such that $a$ fulfill the following ellipticity condition
\bel{ell2}
\sum_{i,j=1}^n a_{i,j}(\mu,\eta) \xi_i \xi_j \geq \kappa(\mu,\eta) |\xi|^2, \quad 
\mbox{for each $\mu\in\R,\ \eta\in\R^n,\  \xi=(\xi_1,\ldots,\xi_n) \in \R^n$}.
\ee
We set also  the constant parameters $\lambda,\tau\in\R$ and $\omega\in\mathbb S^{n-1}:=\{x\in\R^n:\ |x|=1\}$. Then we consider  the following boundary value problem
\bel{eq1}
\left\{
\begin{array}{ll}
-\sum_{i,j=1}^n \partial_{x_i} 
\left( a_{i,j}(u(x),\nabla u(x)) \partial_{x_j} u(x) \right)=0  & \mbox{in}\ \Omega ,
\\
u(x)=\lambda+\tau x\cdot\omega &x\in\partial\Omega.
\end{array}
\right.
\ee

Under the above assumptions, we prove in Proposition \ref{p1} that, for every $\lambda\in\R$, there exists $\epsilon_\lambda>0$ depending only on  $a$, $\lambda$ and $\Omega$  such that for all $\tau\in(-\epsilon_\lambda,\epsilon_\lambda)$,  the problem \eqref{eq1}  admits a unique  solution $u_{\lambda,\omega,\tau}\in  C^{2+\alpha}(\overline{\Omega})$. Then, we associate with problem \eqref{eq1}  the  measurement of the flux at the boundary, generated by the Dirichlet excitation $\lambda+\tau x\cdot\omega$, given by
$$\partial_{\nu_a} u_{\lambda,\omega,\tau}(x)=\sum_{i,j=1}^na_{i,j}(u_{\lambda,\omega,\tau}(x),\nabla u_{\lambda,\omega,\tau}(x))\partial_{x_j}u_{\lambda,\omega,\tau}(x)\nu_i(x),\quad \lambda\in\R,\ \tau\in(-\epsilon_{\lambda},\epsilon_{\lambda}),\ x\in\partial\Omega,$$
with  $\nu(x)=(\nu_1(x),\ldots,\nu_n(x))$ the outward unit normal vector to $\partial\Omega$ computed at $x \in \partial\Omega$.
Fixing $x_1,\ldots,x_m\in\partial\Omega$, with $m\leq n$, we study the  determination of the quasilinear term $a$   from some knowledge of 
\bel{data}\partial_{\nu_a} u_{\lambda,\omega,\tau}(x_j),\quad \lambda\in\R,\ \tau\in(-\epsilon_\lambda,\epsilon_\lambda),\  j=1,\ldots,m.\ee
 More precisely, we are looking for both uniqueness and stability result for this inverse problem.

The inverse problem under investigation in this article can be associated with different physical problems described by quasilinear equations of the form \eqref{eq1}. This includes the determination of the conductivity of a medium, that depend on the voltage and the current, from measurements of the current flux generated by  voltage associated with our class of affine Dirichlet data. Such problem can be considered for  physical models where the transfer from voltage to current density can not be described by the classical Ohm's law but some more general nonlinear expression and it includes problems of viscous flows \cite{GR} or plasticity phenomena \cite{CS}.
We can also mention problems of heat conduction, associated with \eqref{eq1},  where the goal is to determine the thermal conductivity of the medium, associated with the quasilinear term $a$, that depends simultaneously on the temperature and the thermal gradient (see e.g. \cite{Al,BBC,Ca2,ZF}).

The determination of a quasilinear term appearing in a non-linear elliptic equation has received an increasing interest among the mathematical community and several works have been devoted to the study of this class of inverse problems. The first works  in that direction have been devoted to numerical and theoretical study of the determination of conductivities depending only on the solutions from overspecified data in  \cite{Ca1,EPS1,Ku,PR}. For problems more closely related to the one under consideration in this article, one can refer to the works \cite{EPS2,MU,Sh} where the unique determination of  quasilinear
terms has been addressed from data given by the full knowledge of the Dirichlet-to-Neumann map associated with \eqref{eq1} (i.e. Neumann boundary measurements on the whole boundary $\partial\Omega$ of the solutions of the equation \eqref{eq1} for all possible Dirichlet excitations). In the more recent works \cite{C,K}, the authors addressed the stability issue for this last problem for quasilinear terms depending only on the solution from some restriction of the Dirichlet-to-Neumann map associated with \eqref{eq1}. Similar problems have been investigated for more general class of quasilinear terms depending also on the space variable in the works \cite{CF,CFKKU,KKU,Sun1,SuUh} among which the most general one can be found in \cite{CFKKU} where the authors addressed the open problem of determining quasilinear terms depending simultaneously on the solutions, the gradient of the solutions and the space variable. Finally, we mention the works \cite{FO20,ImYa,IN,IS,KrU1,KrU2,LLLS} devoted to the determination of semilinear terms by using the first order linearization technique of \cite{Is1} as well as the higher order linearization initiated by \cite{KLU}.

All the above mentioned results, that are directly connected to our inverse problem, have been stated with measurements of the flux at the boundary associated with Dirichlet excitations lying on an open subset of an infinite dimension space. In addition, the measurements under consideration for the determination of  quasilinear terms have always been restricted to at least an open set of the boundary $\partial\Omega$ of the domain and none of these works addressed the determination of matrix valued quasilinear terms depending simultaneously on the solutions and the gradient of the solutions. The main goal of the present article is to study the determination of general class of matrix valued quasilinear terms depending simultaneously on the solutions and the gradient of the solutions and to make two important restrictions on the data under consideration for this class of inverse problems: 1) Restrict the Dirichlet data to an open set of  affine functions of $\R^n$ taking values in $\R$  which is a space of dimension $n+1$; 2) Localize the measurements to at most $n$ points  of the boundary $\partial\Omega$.

\section{Main results}
In Proposition \ref{p1} we  prove that, for every $\lambda\in\R$ and every $\omega\in\mathbb{S}^{n-1}$, there exists $\epsilon_\lambda>0$ depending only on $a$, $\Omega$, $\lambda$, such that for $|\tau|<\epsilon_\lambda$, the problem \eqref{eq1} admits a unique solution $u_{\lambda,\omega,\tau}\in  C^{2+\alpha}(\overline{\Omega})$. Using this property we can state our first result, which will be a uniqueness result. For this purpose, we will consider quasilinear terms $a\in C^{2+\ell}(\R\times\R^n;\R^{n\times n})$, $\ell\in\mathbb N$, having an orthonormal basis of eigenfunctions depending only on the solution $u$ of the equation \eqref{eq1}. More precisely, we assume that there exist $p\in\mathbb N$, $m_1,\ldots,m_p\in\mathbb N$, with $m_1+\ldots+m_p=n$, 
$\gamma_k\in C^{2+\ell}(\R\times\R^n;(0,+\infty))$, $k=1,\ldots,p$, $U\in  C^{2+\ell}(\R;\R^{n\times n})$ such that
\bel{a1}a(\mu,\eta)=U(\mu)\textrm{Diag}(\underbrace{\gamma_1(\mu,\eta),\ldots,\gamma_1(\mu,\eta)}_{m_1\ \textrm{times}},\gamma_2(\mu,\eta),\ldots,\gamma_p(\mu,\eta))U(\mu)^T,\quad (\mu,\eta)\in\R\times\R^n,\ee
 with
$$0<\gamma_1(\mu,\eta)<\ldots<\gamma_{p}(\mu,\eta),\quad U(\mu)^TU(\mu)=U(\mu)U(\mu)^T=\textrm{Id}_{\R^{n\times n}},\quad (\mu,\eta)\in\R\times\R^n.$$
Here $\textrm{Id}_{\R^{n\times n}}$ stands for the identity matrix of size $n$.\\

In addition to this condition, we make a geometrical assumption that can be stated as follows:\\
(H) There exists an orthonormal basis $\{e_1',\ldots,e_n'\}$ of $\R^n$ and $n$ points $x_1,\ldots,x_n\in\partial\Omega$ such that
\bel{as}|\nu(x_j)-e_j'|<\frac{1}{\sqrt{n}},\quad j=1,\ldots,n.\ee
Our first main result is a uniqueness result that can be stated as follows.
\begin{theorem}\label{t1}  Fix $\ell\in\mathbb N$ and, for $j=1,2$, let $a^j\in C^{2+\ell}(\R\times\R^n;\R^{n\times n})$ satisfy \eqref{ell1}-\eqref{ell2} as well as condition \eqref{a1}. Assume also that condition \emph{(H)} is fulfilled and fix $\{e_1',\ldots,e_n'\}$ an orthonormal basis of $\R^n$ and $n$ points $x_1,\ldots,x_n\in\partial\Omega$ such that \eqref{as} holds true.
Fix $R>0$ and consider $u^j_{\lambda,\omega,\tau}\in C^{2+\alpha}(\overline{\Omega})$, $\lambda\in(-R,R)$, $\omega\in\mathbb S^{n-1}$ and $\tau\in(-\epsilon_\lambda,\epsilon_\lambda)$, the unique solution of \eqref{eq1} with $a=a^j$. Then, for every $\epsilon\in(0,\epsilon_\lambda)$, the condition
\bel{t1b}\partial_{\nu_{a^1}} u_{\lambda,\omega,\tau}^1(x_k)=\partial_{\nu_{a^2}} u_{\lambda,\omega,\tau}^2(x_k),\quad \lambda\in[-R,R],\ \omega\in\mathbb S^{n-1},\ \tau\in(-\epsilon,\epsilon),\ k=1,\ldots,n\ee
implies that 
\bel{t1c}D_\eta^{k} a^1(\lambda,0)=D_\eta^{k} a^2(\lambda,0),\quad \lambda\in[-R,R],\ k=0,\ldots,\ell-1.\ee
\end{theorem} 

As a direct consequence of this result we can prove the following.

\begin{cor}\label{c1} Let the condition of Theorem \ref{t1} be fulfilled for $a^j\in C^{\infty}(\R\times\R^n;\R^{n\times n})$, $j=1,2$, and assume that, for all $\lambda\in[-R,R]$ the map $\R^n\ni\eta\mapsto a^j(\lambda,\eta)\in \R^{n\times n}$ is  analytic. Then,  for every $\epsilon\in(0,\epsilon_\lambda)$ and every $x_1,\ldots,x_n\in\partial\Omega$ such that \eqref{as} holds true, the condition \eqref{t1b} implies that $a^1=a^2$ on $[-R,R]\times\R^n$.\end{cor}
We can also improve the result of Theorem \ref{t1} with measurements restricted at one arbitrary chosen point of $\partial\Omega$  by considering more specific class of quasilinear terms of the form
\bel{cond}a(\mu,\eta)=\gamma(\mu,\eta)\textrm{Id}_{\R^{n\times n}},\quad (\mu,\eta)\in\R\times\R^n,\ee
where $\gamma\in  C^{3}(\R\times\R^n;(0,+\infty))$.

\begin{cor}\label{c2} Let the condition of Theorem \ref{t1} be fulfilled  and assume that, for $j=1,2$, $a^j\in C^{2+\ell}(\R\times\R^n;\R^{n\times n})$ takes the form \eqref{cond} with $\gamma=\gamma^j\in  C^{2+\ell}(\R\times\R^n;\R_+)$. Then, for any arbitrary chosen $x_0\in\partial\Omega$ and  $\epsilon\in(0,\epsilon_\lambda)$, the condition 
\bel{c2a}\partial_{\nu_{a^1}} u_{\lambda,\omega,\tau}^1(x_0)=\partial_{\nu_{a^2}} u_{\lambda,\omega,\tau}^2(x_0),\quad \lambda\in[-R,R],\ \omega\in\mathbb S^{n-1},\ \tau\in(-\epsilon,\epsilon)\ee
implies \eqref{t1c}.\end{cor}

For our stability results, let us first mention that in Proposition \ref{p1} we will prove that for $a\in C^{2+\ell}(\R\times\R^n;\R^{n\times n})$, $\ell\in\mathbb N$, the map $(-\epsilon,\epsilon)\ni \tau\mapsto u_{\lambda,\omega,\tau}\in  C^{2+\alpha}(\overline{\Omega})$ is lying in $C^\ell$. Then, we will consider the stable determination of the quasilinear term $a$ from the knowledge of the data
$$\partial_\tau^k\partial_{\nu_{a^1}} u_{\lambda,\omega,\tau}^1(x_j)|_{\tau=0},\quad \lambda\in\R,\ \omega\in\mathbb S^{n-1},\ k=1,\ldots,\ell,\ j=1,\ldots,n.$$
Note that such data are included in the data under consideration in Theorem \ref{t1}. Our stability result can be stated as follows.

\begin{theorem}\label{t2}  Let the condition of Theorem \ref{t1} be fulfilled with $a^j\in C^{3+\ell}(\R\times\R^n;\R^{n\times n})$, $j=1,2$, and assume that there exist $\gamma^j\in C^{3+\ell}(\R\times\R^n;(0,+\infty))$ and $b^j\in C^{3+\ell}(\R;\R^{n\times n})$ such that
\bel{t2aa} a^{j}(\mu,\eta)=\gamma^j(\mu,\eta)b^j(\mu),\quad \gamma^j(\mu,0)=1,\quad (\mu,\eta)\in\R\times\R^n.\ee
Assume also that there exists $\psi\in C(\R_+;\R_+)$ an increasing function such that
\bel{t2a} \abs{\partial_\mu^k\partial_\eta^\alpha a^j(\mu,0)}_{\R^{n\times n}}\leq \psi(|\mu|),\quad \mu\in\R,\ k\in\mathbb N,\ \alpha\in\mathbb N^n,\ k+|\alpha|\leq 3+\ell.\ee
Consider $u_{\lambda,\omega,\tau}^j$, $j=1,2$, the solution of \eqref{eq1} with $a=a^j$ and $x_1,\ldots,x_n\in\partial\Omega$ satisfying condition \eqref{t1b}.
Then the map $\tau\mapsto \partial_{\nu_{a^j}}u_{\lambda,\omega,\tau}^j(x_i)$, $j=1,2$, $i=1,\ldots,n$, is lying in $C^{\ell}(\R)$ and, for every $R>0$, there exists a constant $C_1>0$ depending only on $n$, $\Omega$ and $x_1,\ldots,x_n$ such that
\bel{t2bb}\norm{b^1-b^2}_{L^\infty((-R,R);\R^{n\times n})}\leq C_1\max_{j=1,\ldots,n}\sup_{\omega\in \mathbb S^{n-1}}\sup_{\lambda\in[-R,R]}\abs{\partial_\tau(\partial_{\nu_{a^1}} u_{\lambda,\omega,\tau}^1(x_j)-\partial_{\nu_{a^2}} u_{\lambda,\omega,\tau}^2(x_j))|_{\tau=0}}.\ee
In addition,  there exists   $C_2>0$ depending on $x_1,\ldots,x_n$, $\kappa$, $\psi$, $R$, $n$ and $\Omega$ such that, for all $N=1,\ldots,\ell-1$, we have
\bel{t2b} \begin{aligned}&\sup_{\lambda\in[-R,R]}\abs{D_\eta^{N} \gamma^1(\lambda,0)-D_\eta^{N} \gamma^2(\lambda,0)}_{T_N}\\
&\leq C_2\sum_{k=1}^{N+1}\left(\max_{j=1,\ldots,n}\sup_{\omega\in \mathbb S^{n-1}}\sup_{\lambda\in[-R,R]}\abs{\partial_\tau^k(\partial_{\nu_{a^1}} u_{\lambda,\omega,\tau}^1(x_j)-\partial_{\nu_{a^2}} u_{\lambda,\omega,\tau}^2(x_j))|_{\tau=0}}\right)^{\frac{3^k}{3^{N+1}}},\end{aligned}\ee
where $T_N$ denotes the space of tensors of rank $N$ of $\R^n$.
\end{theorem} 
In a similar way to Corollary \ref{c2}, we can improve this stability estimate with measurement at one point when the quasilinear term $a$ takes the form \eqref{cond}.

\begin{cor}\label{c3} Let the condition of Theorem \ref{t2} be fulfilled  and assume that, for $j=1,2$, $a^j\in C^{3+\ell}(\R\times\R^n;\R^{n\times n})$ takes the form \eqref{cond} with $\gamma=\gamma^j\in  C^{3+\ell}(\R\times\R^n;\R_+)$. Then, for any arbitrary chosen $x_0\in\partial\Omega$ and  $R>0$, there exists a constant $C_1>0$ depending only on $n$  such that
\bel{c3a}\sup_{\lambda\in[-R,R]}\abs{\gamma^1(\lambda,0)-\gamma^2(\lambda,0)}\leq C_1\sup_{\omega\in \mathbb S^{n-1}}\sup_{\lambda\in[-R,R]}\abs{\partial_\tau(\partial_{\nu_{a^1}} u_{\lambda,\omega,\tau}^1(x_0)-\partial_{\nu_{a^2}} u_{\lambda,\omega,\tau}^2(x_0))|_{\tau=0}}.\ee
In addition,  there exists   $C_2>0$ depending on $x_0$, $\kappa$, $\psi$, $R$, $n$ and $\Omega$ such that, for all $N=1,\ldots,\ell-1$, we have
\bel{c3b} \begin{aligned}&\sup_{\lambda\in[-R,R]}\norm{D_\eta^{N} \gamma^1(\lambda,0)-D_\eta^{N} \gamma^2(\lambda,0)}_{T_N}\\
&\leq C_2\sum_{k=1}^{N+1}\left(\sup_{\omega\in \mathbb S^{n-1}}\sup_{\lambda\in[-R,R]}\abs{\partial_\tau^k(\partial_{\nu_{a^1}} u_{\lambda,\omega,\tau}^1(x_0)-\partial_{\nu_{a^2}} u_{\lambda,\omega,\tau}^2(x_0))|_{\tau=0}}\right)^{\frac{3^k}{3^{N+1}}}.\end{aligned}\ee
\end{cor}

As far as we know, in all other relevant results related to our inverse problem, that can be found for instance in the articles  \cite{C,EPS2,K,MU,Sh}, the determination of the quasilinear term $a$  has been considered from Neumann boundary measurements restricted to an open subset of $\partial\Omega$ associated with Dirichlet excitations lying in an infinite dimensional space. In Theorem \ref{t1} and \ref{t2}, we improve these results by restricting the Dirichlet excitations to the space of  affine functions of $\R^n$ taking values in $\R$, which is a space of dimension $n+1$, and we consider Neumann measurements restricted to at most $n$ points for the determination of general quasilinear terms depending simultaneously on the solutions and the gradient of the solutions of \eqref{eq1}. Actually, Theorem \ref{t1} and \ref{t2} seem to be the first results of determination of nonlinear terms appearing in an elliptic equation from such important restriction of the Cauchy data. It seems also that these results are the first addressing the determination of general matrix valued quasilinear terms, depending simultaneously on the solutions and the gradient of the solutions, that are not taking the form  \eqref{cond}. In that sense our results provide an important improvement, in terms of restriction of the data and generalization, for the resolution of this class of inverse problems associated with non-linear terms independent of the space variable. In addition, the Dirichlet excitations are given by explicit class of affine functions and we derive in Theorem \ref{t2} some H\"older stability estimates for our inverse problems. This makes our theoretical results more flexible for some potential applications to the numerical reconstruction of these class of quasilinear terms by different iterative method such as Tikhonov regularization (see e.g. \cite{CY} for more details).

One of the key ingredient in the proof of our main results, is given by the asymptotic properties of the data \eqref{data} as $\tau\to0$. We prove these properties by combining Taylor formula with Lemma \ref{l1} and \ref{l2}. Our approach can be compared with the linearization techniques introduced by \cite{Is1,KLU} and applied for the resolution of most inverse problems of determining nonlinear terms for elliptic equations (see e.g. \cite{CFKKU,C,FO20,ImYa,IN,IS,KrU1,KrU2,LLLS,Sun1,SuUh}). Nevertheless, in contrast to all these works, we do not use this approach in order to transform our inverse problems into a linearized problem stated in terms of density or concentration of products of solutions of linear elliptic equations. Instead, we use this linearization technique in order to derive explicit asymptotic properties of the data \eqref{data}, as $\tau\to0$, and we use such asymptotic properties for solving our inverse problem. As far as we know, this work seems to be the first where the higher order linearization technique is used for such purpose.

As stated in Corollary \ref{c1}, the result of Theorem \ref{t1} leads to the full unique determination of the quasilinear term $a$ under an assumption of analiticity imposed to the quasilinear term $a$ with respect to the variable $\eta\in\R^n$ associated with the gradient of the solutions. This seems to be the first result of full determination of a matrix valued quasilinear term, depending simultaneously on the solutions and the gradient of the solutions. However, our approach works for quasilinear terms $a$ of the form \eqref{a1}. It is not yet clear how such restriction can be removed in the proof of Theorem \ref{t1} and the determination of a general quasilinear term satisfying only conditions \eqref{ell1}-\eqref{ell2} remains an open problem.

Let us remark that Theorem \ref{t1} and \ref{t2} are subjected to the geometrical condition (H) which will be fulfilled for rather general class of bounded open set $\Omega$ (balls, annulus...). In addition, the condition \eqref{as} allows rather flexible choice for the $n$ points $x_1,\ldots,x_n$ of $\partial\Omega$ where the measurements are made.

Let us observe that the results of Theorem \ref{t1} and \ref{t2} are drastically improved in Corollary \ref{c2} and \ref{c3} for quasilinear terms of the form \eqref{cond}, where the measurements can be restricted to any arbitrary point of $\partial\Omega$.

Let us mention that, in a similar way to \cite{K}, we derive in \eqref{t2bb} and \eqref{c3a}, Lipschitz stability estimates with a
constant completely independent of the nonlinear term $a^j$, $j=1,2$. This means that these two estimates are not  conditional stability estimates requiring  \textit{a priori} estimates of the unknown parameter. On the other hand, the stability estimates \eqref{t2b} and 
\eqref{c3b} are H\"older stability estimates requiring the \textit{a priori} estimate \eqref{t2a}. This difference is due to the fact that we derive the estimates \eqref{t2b} and  \eqref{c3b} by mean of an iterative process involving some results of interpolation.

This article is organized as follows. In Section 3, we recall some properties of \eqref{eq1} including the unique existence of solutions and we derive some explicit asymptotic properties of the data \eqref{data} as $\tau\to0$. Using these properties, we show in Section 4 our uniqueness results stated in Theorem \ref{t1} and Corollary \ref{c2}. Section 5, will be devoted to the proof of the stability estimates stated in Theorem \ref{t2} and Corollary \ref{c3}.
\section{Preliminary properties}
In this section we consider several properties including the unique existence of solutions of \eqref{eq1} as well as some asymptotic properties of the data under consideration in our inverse problems.
The unique existence  of solutions of \eqref{eq1}  can be stated as follows.

\begin{prop}\label{p1} We assume that $a\in C^{2+\ell}(\R\times\R^n)$, with $\ell\in\mathbb N$, satisfies  the conditions \eqref{ell1}-\eqref{ell2}.  For all $\lambda\in\R$ and all $\omega\in\mathbb S^{n-1}$, there exists $\epsilon_{\lambda}>0$ depending on  $a$,  $\Omega$,  and $\lambda$,  such that for all $\tau\in(-\epsilon_{\lambda},\epsilon_{\lambda})$, problem \eqref{eq1} admits a unique solution $u_{\lambda,\omega,\tau}\in C^{2+\alpha}(\overline{\Omega})$. In addition, the map $\tau\mapsto u_{\lambda,\omega,\tau}$ is lying in $ C^{\ell}((-\epsilon_{\lambda},\epsilon_{\lambda});C^{2+\alpha}(\overline{\Omega}))$. 
\end{prop}

\begin{proof} Note first that the comparison principles for such class of quasilinear equations of divergence form (see e.g. \cite[Theorem 9.5]{GT}) guaranty the uniqueness of solutions of \eqref{eq1} lying in $C^1(\overline{\Omega})$. Therefore, we only need to show the existence of such solutions and the last statement of the proposition. We follow the line of \cite[Proposition 2.1]{K} (see also \cite[Theorem B.1.]{CFKKU}), where a similar boundary value problem has been studied for more specific class of quasilinear terms $a$. 

Splitting $u_{\lambda,\omega,\tau}$ into two terms $u_{\lambda,\omega,\tau}=\lambda+v_{\lambda,\omega,\tau}$, we find that $v_{\lambda,\omega,\tau}$ solves
\bel{eq4}
\left\{
\begin{array}{ll}
-\sum_{i,j=1}^n \partial_{x_i} 
\left( a_{i,j}(\lambda+v_{\lambda,\omega,\tau}(x),\nabla  v_{\lambda,\omega,\tau}(x)) \partial_{x_j} v_{\lambda,\omega,\tau}(x)\right)=0  & x\in\Omega ,
\\
v_{\lambda,\omega,\tau}(x)=\tau x\cdot\omega &x\in\partial\Omega,
\end{array}
\right.
\ee
Therefore, we only need to prove that there exists $\epsilon_{\lambda}>0$ depending on  $a$,  $\lambda$, $\Omega$, such that, for all $\tau\in(-\epsilon_\lambda,\epsilon_\lambda)$, the problem \eqref{eq4} admits a  solution $v_{\lambda,\omega,\tau}\in C^{2+\alpha}(\overline{\Omega})$ satisfying
\bel{TA1v}\norm{v_{\lambda,\omega,\tau}}_{ C^{2+\alpha}(\overline{\Omega})}\leq C|\tau|.\ee
We define the map $\mathcal K$ from $ C^{2+\alpha}(\partial\Omega)\times C^{2+\alpha}(\overline{\Omega})$ to the space $ C^{\alpha}(\overline{\Omega})\times C^{2+\alpha}(\partial\Omega)$ by
$$\mathcal K: (f,v)\mapsto\left(-\sum_{i,j=1}^n \partial_{x_i} 
\left( a_{i,j}(\lambda+v,\nabla v) \partial_{x_j} v \right), v_{|\partial\Omega}-f\right).$$
 Using the fact that $a\in C^{2+\ell}(\R\times\R^n)$ and applying \cite[Theorem  A.7, A.8]{H},  for any $v\in  C^{2+\alpha}(\overline{\Omega})$, we have $x\mapsto \partial_\mu^k\partial_\eta^\alpha a_{i,j}(\lambda+v(x),\nabla v(x))\in  C^{1+\alpha}(\overline{\Omega})$, $i,j=1,\ldots,n$, $|\alpha|+k\leq\ell$. In view of \cite[Theorem  A.7]{H}, we obtain 
$$x\mapsto-\sum_{i,j=1}^n \partial_{x_i} 
\left( \partial_\mu^k\partial_\eta^\alpha a_{i,j}(\lambda+v(x),\nabla v(x)) \partial_{x_j} v(x)\right)\in  C^{\alpha}(\overline{\Omega}),\quad v\in  C^{2+\alpha}(\overline{\Omega})$$
and the map
$$ C^{2+\alpha}(\overline{\Omega})\ni v\mapsto -\sum_{i,j=1}^n \partial_{x_i}\left( a_{i,j}(\lambda+v(x),\nabla v(x)) \partial_{x_j} v\right)\in  C^{\alpha}(\overline{\Omega})$$
is $ C^\ell$. It follows  that the map $\mathcal K$ is  $ C^\ell$ from $ C^{2+\alpha}(\partial\Omega)\times C^{2+\alpha}(\overline{\Omega})$ to the space $ C^{\alpha}(\overline{\Omega})\times C^{2+\alpha}(\partial\Omega)$. Moreover, we have $\mathcal K(0,0)=(0,0)$ and
$$\partial_v\mathcal K(0,0)w=\left(-\sum_{i,j=1}^n  
 a_{i,j}(\lambda,0) \partial_{x_i}\partial_{x_j} w, w_{|\partial\Omega}\right).$$
In view of \cite[Theorem 6.8]{GT}, for any $(F,f)\in  C^{\alpha}(\overline{\Omega})\times C^{2+\alpha}(\partial\Omega)$, the following linear boundary value problem
$$
\left\{
\begin{array}{ll}
-\sum_{i,j=1}^n 
 a_{i,j}(\lambda,0) \partial_{x_i} \partial_{x_j} w =F  & \mbox{in}\ \Omega ,
\\
w=f &\mbox{on}\ \partial\Omega.
\end{array}
\right.$$
admits a unique solution $w\in  C^{2+\alpha}(\overline{\Omega})$ satisfying
$$\norm{w}_{ C^{2+\alpha}(\overline{\Omega})}\leq C(\norm{F}_{ C^{\alpha}(\overline{\Omega})}+\norm{f}_{ C^{2+\alpha}(\partial\Omega)}),$$
with $C>0$ depending only on  $a$, $\lambda$, $\Omega$. Thus, $\partial_v\mathcal K(0,0)$ is an isomorphism from $ C^{2+\alpha}(\overline{\Omega})$ to $ C^{\alpha}(\overline{\Omega})\times C^{2+\alpha}(\partial\Omega)$ and,  the implicit function theorem implies that there exists $r>0$ depending on  $a$, $\lambda$, $\Omega$, and a $ C^\ell$ map $\psi$ from $ B_{r}:=\{g\in C^{2+\alpha}(\partial\Omega):\ \norm{g}_{C^{2+\alpha}(\partial\Omega)}<r\}$ to $ C^{2+\alpha}(\overline{\Omega})$, such that, for all $f\in  B_{r}$, we have 
$\mathcal K(f,\psi(f))=(0,0)$. 
Thus, setting $C_0>0$ defined by
$$C_0=\sup_{x\in\overline{\Omega}}|x|+n$$
we have
$$\sup_{\omega\in\mathbb S^{n-1}}\norm{x\cdot\omega}_{C^{2+\alpha}(\partial\Omega)}\leq C_0$$
and fixing $\epsilon_{\lambda}=\frac{r}{C_0}$, for all $\tau\in  (-\epsilon_{\lambda},\epsilon_{\lambda})$, $v_{\lambda,\omega,\tau}=\psi(\tau x\cdot\omega)$ is a solution of \eqref{eq4}.  Using the fact that $\psi$ is $ C^\ell$ from $B_{r}$ to $ C^{2+\alpha}(\overline{\Omega})$ and $\psi(0)=0$, we obtain \eqref{TA1v} and we deduce that the map $\tau\mapsto u_{\lambda,\omega,\tau}=\lambda+\psi(\tau x\cdot\omega)$ is lying in $ C^\ell((-\epsilon_{\lambda},\epsilon_{\lambda});C^{2+\alpha}(\overline{\Omega}))$. This completes the proof of the proposition.\end{proof}

Using the results of Proposition \ref{p1}, we will give some asymptotic properties of the data $\partial_{\nu_a}u_{\lambda,\omega,\tau}(x)$, $x\in\partial\Omega$ as $\tau\to0$. In view of Taylor formula, provided that the map $\tau\mapsto\partial_{\nu_a}u_{\lambda,\omega,\tau}(x)$ is sufficiently smooth on a neighborhood of $\tau=0$, such asymptotic properties will be given by the expression
\bel{dev}\partial_\tau^k(\partial_{\nu_a}u_{\lambda,\omega,\tau}(x))|_{\tau=0},\quad k\in\mathbb N,\quad x\in\partial\Omega.\ee
Applying the results of Proposition \ref{p1}, we will show such a smoothness property of the map $\tau\mapsto\partial_{\nu_a}u_{\lambda,\omega,\tau}(x)$ and we will derive some explicit formula for \eqref{dev}. Our approach can be compared with the higher order linearization technique considered in a different way for  similar class of inverse problems for nonlinear equations (see e.g. \cite{CFKKU,FO20,IN,IS,KrU1,KrU2,LLLS,Sun1,SuUh}).

In view of Proposition \ref{p1}, if  $a\in C^{3}(\R\times\R^n;\R^{n\times n})$, for all $\lambda\in\R$, $\omega\in\mathbb S^{n-1}$ and all $\tau\in(-\epsilon_{\lambda},\epsilon_{\lambda})$, problem \eqref{eq1} admits a unique solution $u_{\lambda,\omega,\tau}\in C^{2+\alpha}(\overline{\Omega})$ and the map $\tau\mapsto u_{\lambda,\omega,\tau}$ is lying in $ C^{1}((-\epsilon_{\lambda},\epsilon_{\lambda});C^{2+\alpha}(\overline{\Omega}))$. 
 Using this property and the results of Proposition \ref{p1}, we obtain the following result.

\begin{lem}\label{l1} Assume that $a\in C^{3}(\R\times\R^n;\R^{n\times n})$ satisfies  the conditions \eqref{ell1}-\eqref{ell2}. Then, for all $\lambda\in\R$ and all $\omega\in\mathbb S^{n-1}$, the map $\tau\mapsto \partial_{\nu_a}  u_{\lambda,\omega,\tau}$ is lying in $C^1((-\epsilon_{\lambda,h},\epsilon_{\lambda,h});C^{1+\alpha}(\partial\Omega))$ and we have
\bel{l1a}\partial_\tau(\partial_{\nu_a}u_{\lambda,\omega,\tau}(x))|_{\tau=0} =a(\lambda,0)\omega\cdot \nu(x),\quad x\in\partial\Omega.\ee
\end{lem}
\begin{proof} Fix $\lambda\in\R$. In view of Proposition \ref{p1}, we have  $\tau\mapsto u_{\lambda,\omega,\tau}\in C^1((-\epsilon_{\lambda},\epsilon_{\lambda});C^{2+\alpha}(\overline{\Omega}))$ which implies that $\tau\mapsto \partial_{\nu_a}  u_{\lambda,\omega,\tau}\in C^1((-\epsilon_{\lambda,h},\epsilon_{\lambda,h});C^{1+\alpha}(\partial\Omega))$. In addition, using the fact that $u_{\lambda,\omega,\tau}|_{\tau=0}=\lambda$, we obtain $\nabla u_{\lambda,\omega,\tau}|_{\tau=0}\equiv 0$ and it follows
 \bel{l1b}\partial_\tau(\partial_{\nu_a}u_{\lambda,\omega,\tau})|_{\tau=0}=\sum_{i,j=1}^na_{i,j}(\lambda,0)\partial_{x_j}\partial_\tau u_{\lambda,\omega,\tau}|_{\tau=0}\nu_i=a(\lambda,0)\nabla \partial_\tau u_{\lambda,\omega,\tau}|_{\tau=0}\cdot \nu,\ \textrm{ on $\partial\Omega$}.\ee On the other hand, we have
$$\partial_\tau \sum_{i,j=1}^n \partial_{x_i} 
\left( a_{i,j}(u_{\lambda,\omega,\tau},\nabla u_{\lambda,\omega,\tau}) \partial_{x_j} u_{\lambda,\omega,\tau} \right)|_{\tau=0}=\sum_{i,j=1}^n 
 a_{i,j}(\lambda,0) \partial_{x_i} \partial_{x_j}\partial_\tau u_{\lambda,\omega,\tau}|_{\tau=0}$$
and $\partial_\tau u_{\lambda,\omega,\tau}|_{\tau=0}=x\cdot\omega$ on $\partial\Omega$. Therefore, $w_\lambda=\partial_\tau u_{\lambda,\omega,\tau}|_{\tau=0}$ solves the boundary value problem
\bel{eq3}
\left\{
\begin{array}{ll}
-\sum_{i,j=1}^n 
 a_{i,j}(\lambda,0) \partial_{x_i} \partial_{x_j} w_\lambda =0  & \mbox{in}\ \Omega ,
\\
w_\lambda(x)=x\cdot\omega &x\in \partial\Omega.
\end{array}
\right.
\ee
Recalling that $x\mapsto x\cdot\omega$ solves the problem \eqref{eq3} and using the uniqueness of the solution of this problem, we deduce that 
$$\partial_\tau u_{\lambda,\omega,\tau}|_{\tau=0}(x)=w_\lambda(x)=x\cdot\omega,\quad x\in\overline{\Omega}.$$
Combining this with \eqref{l1b}, we deduce  \eqref{l1a}.

\end{proof}

Applying Lemma \ref{l1} we obtain the asymptotic expansion of $\partial_{\nu_a}u_{\lambda,\omega,\tau}(x)$, $x\in\partial\Omega$, at order 1 in $\tau$ as $\tau\to0$. In order to extend this asymptotic property, we will need explicit formula for \eqref{dev} with $k\geq2$.
 For this purpose, let us assume that $a\in C^{2+N}(\R\times\R^n;\R^{n\times n})$, $N\geq2$. In view of Proposition \ref{p1}, this assumption implies that the map $\tau\mapsto\partial_{\nu_a}u_{\lambda,\omega,\tau}(x)$, $x\in\partial\Omega$, is lying  in $C^N(-\epsilon_{\lambda},\epsilon_{\lambda})$ and we can consider  the following boundary value problem

\bel{eq6}
\left\{
\begin{array}{ll}
-\sum_{i,j=1}^n 
 a_{i,j}(\lambda,0) \partial_{x_i} \partial_{x_j} y_\lambda =\nabla\cdot K_\lambda  & \mbox{in}\ \Omega ,
\\
y_\lambda=0 &\mbox{on}\ \partial\Omega,
\end{array}
\right.
\ee
with
$$K_\lambda=\partial_{\tau}^N a(u_{\lambda,\omega,\tau},\nabla u_{\lambda,\omega,\tau}) \nabla u_{\lambda,\omega,\tau} |_{\tau=0}-a(\lambda,0) \nabla \partial_{\tau}^Nu_{\lambda,\omega,\tau}|_{\tau=0}.$$

Then, using some arguments of Lemma \ref{l1}, we can show by iteration the following.

\begin{lem}\label{l2}  Let $a\in C^{2+N}(\R\times\R^n;\R^{n\times n})$, $N\geq2$, and consider the solution $y_\lambda$ of \eqref{eq6}. Then, , for all $\lambda\in\R$ and all $\omega\in\mathbb S^{n-1}$, the map $\tau\mapsto \partial_{\nu_a}  u_{\lambda,\omega,\tau}$ is lying in $C^N((-\epsilon_{\lambda,h},\epsilon_{\lambda,h});C^{1+\alpha}(\partial\Omega))$ and we find
\bel{l2b} \begin{aligned}&\partial_\tau^N(\partial_{\nu_a}u_{\lambda,\omega,\tau}(x))|_{\tau=0}\\
&=N D_\eta^{N-1}a(\lambda,0)(\omega,\ldots,\omega)\omega\cdot\nu(x)+H_\lambda(x)+\sum_{i,j=1}^n 
 a_{i,j}(\lambda,0)  \partial_{x_j}y_\lambda\nu_i(x),\quad x\in\partial\Omega\end{aligned}\ee
where  the expression $H_\lambda$ and  $y_\lambda$ depend  only on $\omega$, $\lambda$, $\Omega$ and $\partial_\mu^j D^k_{\eta}a(\lambda,0)$,  $k=0,\ldots,N-2$ and $k+j\leq N-1$. 
\end{lem}
\begin{proof} Let us first observe that the uniqueness of solutions of problem \eqref{eq6} implies that $y_\lambda=\partial_\tau^Nu_{\lambda,\omega,\tau}|_{\tau=0}$ with a differentiation in $\tau$ considered in the sense of functions taking values in $C^{2+\alpha}(\overline{\Omega})$. Combining this with the  fact that $\nabla u_{\lambda,\omega,\tau}|_{\tau=0}\equiv 0$ and $\partial_\tau u_{\lambda,\omega,\tau}|_{\tau=0}=x\cdot\omega$, one can  check by iteration that \eqref{l2b}  holds true with $H_\lambda$  depending  only on $\omega$, $\lambda$, $\Omega$ and $\partial_\mu^j D^k_{\eta}a(\lambda,0)$,  $k=0,\ldots,N-2$ and $k+j\leq N-1$. In the same way, one can check that $y_\lambda$ depends only on $\omega$, $\lambda$, $\Omega$ and $\partial_\mu^j D^k_{\eta}a(\lambda,0)$,   $k+j\leq N-1$.
Therefore, we only need to show that $y_\lambda$ is independent of $D_\eta^{N-1}a(\lambda,0)$. For this purpose, notice that the expression $K_\lambda$ appearing in \eqref{eq6} takes the form
$$K_\lambda=N D_\eta^{N-1}a(\lambda,0)(\omega,\ldots,\omega)\omega+K_\lambda'$$
with $K_\lambda'$ independent of $D_\eta^{N-1}a(\lambda,0)$. Therefore, we have
$$\nabla\cdot K_\lambda=\nabla\cdot K_\lambda'$$
which proves that $\nabla\cdot K_\lambda$ is independent of  $D_\eta^{N-1}a(\lambda,0)$ and the same is true for $y_\lambda$.
\end{proof}

Using these properties we are now in position to complete the proof of our main results.

\section{Proof of the uniqueness results}
\subsection{Proof of Theorem \ref{t1}}

We will show this result by iteration. One of the key ingredient of our proof will be the asymptotic properties of the data $\partial_{\nu_{a^j}}u_{\lambda,\omega,\tau}^j$, $j=1,2$, exhibited in Lemma \ref{l1} and \ref{l2}.

We start by proving that \eqref{t1b}, for some $\epsilon\in(0,\epsilon_\lambda)$, implies \eqref{t1c}, with $k=0$. For this purpose, we fix $\lambda\in[-R,R]$ and applying Lemma \ref{l1}, for all $j=1,\ldots,n$ and all $\omega\in\mathbb S^{n-1}$,  we deduce that
\bel{t1d}a^1(\lambda,0) \omega\cdot\nu(x_j)=\partial_\tau(\partial_{\nu_{a^1}}u_{\lambda,\omega,\tau}^1(x_j))|_{\tau=0}=\partial_\tau(\partial_{\nu_{a^2}}u_{\lambda,\omega,\tau}^2(x_j))|_{\tau=0}=a^2(\lambda,0)\omega\cdot\nu(x_j).\ee
Let us consider $\{e_1,\ldots,e_n\}$ an orthonormal basis of $\R^n$. Fixing $k\in\{1,\ldots,n\}$ and choosing $\omega=e_k$, we deduce from \eqref{t1d} that
$$a^1(\lambda,0)e_k\cdot\nu(x_j)=a^2(\lambda,0)e_k\cdot\nu(x_j),\quad  j,k=1,\ldots,n.$$
We fix
\bel{M}M=\max_{j=1,\ldots,n}|\nu(x_j)-e_j'|\ee
and, in view of \eqref{as}, we have $M<\frac{1}{\sqrt{n}}$. Then, we find
$$\begin{aligned}\abs{(a^1(\lambda,0)e_k-a^2(\lambda,0)e_k)\cdot e_j'}&\leq \abs{(a^1(\lambda,0)e_k-a^2(\lambda,0)e_k)}|\nu(x_j)-e_j'|\\
&\leq M\abs{(a^1(\lambda,0)e_k-a^2(\lambda,0)e_k)},\quad j,k=1,\ldots,n\end{aligned}$$
and, recalling that $\{e_1',\ldots,e_n'\}$ is an orthonormal basis of $\R^n$, we get
$$\begin{aligned}\abs{(a^1(\lambda,0)e_k-a^2(\lambda,0)e_k)}^2&=\sum_{j=1}^n\abs{(a^1(\lambda,0)e_k-a^2(\lambda,0)e_k)\cdot e_j'}^2\\
&\leq nM^2\abs{(a^1(\lambda,0)e_k-a^2(\lambda,0)e_k)}^2,\quad k=1,\ldots,n.\end{aligned}$$
Then, we have 
$$(1-nM^2)\abs{(a^1(\lambda,0)e_k-a^2(\lambda,0)e_k)}^2\leq 0,\quad k=1,\ldots,n$$
and, using the fact that $1-nM^2>0$, we get
$$(a^1(\lambda,0)-a^2(\lambda,0))e_k=0,\quad k=1,\ldots,n$$
which clearly implies \eqref{t1c}, with $k=0$.

Now let us fix $m\in\{0,\ldots,\ell-2\}$ and let us assume that
\bel{t1f}D_\eta^{j} a^1(\lambda,0)=D_\eta^{j} a^2(\lambda,0),\quad \lambda\in[-R,R],\ j=0,\ldots,m\ee
holds true and under this assumption let us show that \eqref{t1b} implies
\bel{t1g}D_\eta^{m+1} a^1(\lambda,0)=D_\eta^{m+1} a^2(\lambda,0),\quad \lambda\in[-R,R].\ee
In light of condition \eqref{ell1}-\eqref{ell2} and \eqref{a1}, we can fix $\gamma_k^j\in C^{2+\ell}(\R\times\R^n;(0,+\infty))$, $k=1,\ldots,p_j$, $j=1,2$, such that 
$$0<\gamma_1^j(\mu,\eta)<\ldots<\gamma_{p_j}^j(\mu,\eta),\quad \sigma(a^j(\mu,\eta))=\{\gamma_1^j(\mu,\eta),\ldots,\gamma_{p_j}^j(\mu,\eta)\},\quad (\mu,\eta)\in\R\times\R^n,$$
where $\sigma(a^j(\mu,\eta))$ denotes the spectrum of the matrix $a^j(\mu,\eta)$. Since $a^j$ fulfills condition  \eqref{a1}, we can find $m_1^j,\ldots,m_{p_j}^j\in\mathbb N$ and $f_{i,j}^j\in C^{2+\ell}(\R;\R^n)$, $i=1,\ldots,p_j$, $j=1,\ldots,m_i^j$ such that, for all $i=1,\ldots,p_j$ and all $(\mu,\eta)\in\R\times\R^n$, we have that
$$ \{f_{i,1}^j(\mu),\ldots,f_{i,m_i^j}^j(\mu)\}\textrm{ is an orthonormal basis of }Ker(a^j(\mu,\eta)-\gamma_i^j(\mu,\eta)\textrm{Id}_{\R^{n\times n}}).$$
Using the fact that \eqref{t1c} holds true for $k=0$, we deduce that $p_1=p_2:=p$, $m_{p_1}^1=m_{p_2}^2:=m_p$ and
$$ \gamma_k^1(\lambda,0)=\gamma_k^2(\lambda,0),\quad \lambda\in[-R,R],\ k=1,\ldots,p.$$
Moreover, for all $\lambda\in[-R,R]$ and all $\eta\in\R^n$, we have
$$\begin{aligned}Span\{f_{k,1}^1(\lambda),\ldots,f_{k,m_k}^1(\lambda)\}&=Ker(a^1(\lambda,0)-\gamma_k^1(\lambda,0)\textrm{Id}_{\R^{n\times n}})\\
&=Ker(a^2(\lambda,0)-\gamma_k^2(\lambda,0)\textrm{Id}_{\R^{n\times n}})\\
&=Span\{f_{k,1}^2(\lambda),\ldots,f_{k,m_k}^2(\lambda)\}\\
&=Ker(a^2(\lambda,\eta)-\gamma_k^2(\lambda,\eta)\textrm{Id}_{\R^{n\times n}}),\quad k=1,\ldots,p.\end{aligned}$$
Therefore, for all $k=1,\ldots,p$, $\lambda\in[-R,R]$ and all $\eta\in\R^n$, $\{f_{k,1}^1(\lambda),\ldots,f_{k,m_k}^1(\lambda)\}$ is an orthonormal basis of the  subspace $$Ker(a^1(\lambda,\eta)-\gamma_k^1(\lambda,\eta)\textrm{Id}_{\R^{n\times n}})=Ker(a^2(\lambda,\eta)-\gamma_k^2(\lambda,\eta)\textrm{Id}_{\R^{n\times n}})$$ of $\R^n$. In addition, applying  \eqref{a1}, we deduce that, for all $\lambda\in[-R,R]$, 
$$\{f_{i,j}^1(\lambda):\ i=1,\ldots,p,\ j=1,\ldots,m_i\}$$
is an orthonormal basis of $\R^n$. Thus, for all $\lambda\in[-R,R]$, $\xi_1,\ldots,\xi_{m+1}\in\R^n$,  we have
\bel{t1h}\begin{aligned}& D_\eta^{m+1}a^1(\lambda,0)(\xi_1,\ldots,\xi_{m+1})\xi_{m+2}-D_\eta^{m+1}a^1(\lambda,0)(\xi_1,\ldots,\xi_{m+1})\xi_{m+2}\\
&=\sum_{i=1}^p\sum_{j=1}^{m_i}D_\eta^{m+1}(\gamma_i^1-\gamma_i^2)(\lambda,0)(\xi_1,\ldots,\xi_{m+1})(\xi_{m+2}\cdot f_{i,j}^1(\lambda))f_{i,j}^1(\lambda).\end{aligned}\ee
In particular, the condition
\bel{t1i}D_\eta^{m+1}\gamma_i^1(\lambda,0)=D_\eta^{m+1}\gamma_i^2(\lambda,0),\quad \lambda\in[-R,R],\ i=1,\ldots,p,\ee
implies \eqref{t1g}. So we are left with the proof of \eqref{t1i}.

We fix $\omega\in\mathbb S^{n-1}$, $\lambda\in[-R,R]$ and we consider the solution $u^j_{\lambda,\omega,\tau}$ of \eqref{eq1} with $a=a^j$. In view of Lemma \ref{l2}, 
$$\partial_\tau^{m+2}(\partial_{\nu_a}u_{\lambda,\omega,\tau}^j(x_i))|_{\tau=0}=(m+2) D_\eta^{m+1}a(\lambda,0)(\omega,\ldots,\omega)\omega\cdot\nu(x_i)+\mathcal M_\lambda^j(x_i),\quad j=1,2,\ i=1,\ldots,n,$$
where $\mathcal M_\lambda^j$, $j=1,2$, depends  only on $\omega$, $\lambda$, $\Omega$ and $\partial_\mu^j D^k_{\eta}a(\lambda,0)$,  $k=0,\ldots,m$ and $k+j\leq m+1$.
In light of \eqref{t1f}, we obtain
$$\mathcal M_\lambda^1(x_i)=\mathcal M_\lambda^2(x_i),\quad i=1,\ldots,n,$$
which implies
\bel{t1k}[D_\eta^{m+1}a^1(\lambda,0)-D_\eta^{m+1}a^2(\lambda,0)](\omega,\ldots,\omega)\omega\cdot\nu(x_i)=0,\quad i=1,\ldots,n,\ \omega\in\mathbb S^{n-1}.\ee
Then, fixing $M$ given by \eqref{M}, we get
$$\begin{aligned}&\abs{([D_\eta^{m+1}a^1(\lambda,0)-D_\eta^{m+1}a^2(\lambda,0)](\omega,\ldots,\omega)\omega\cdot e_i'}\\
&\leq \abs{[D_\eta^{m+1}a^1(\lambda,0)-D_\eta^{m+1}a^2(\lambda,0)])](\omega,\ldots,\omega)\omega}|\nu(x_i)-e_i'|\\
&\leq M\abs{[D_\eta^{m+1}a^1(\lambda,0)-D_\eta^{m+1}a^2(\lambda,0)])](\omega,\ldots,\omega)\omega},\quad i=1,\ldots,n,\ \omega\in\mathbb S^{n-1}\end{aligned}$$
and it follows that
$$\begin{aligned}&\abs{[D_\eta^{m+1}a^1(\lambda,0)-D_\eta^{m+1}a^2(\lambda,0)])](\omega,\ldots,\omega)\omega}^2\\
&=\sum_{i=1}^n\abs{([D_\eta^{m+1}a^1(\lambda,0)-D_\eta^{m+1}a^2(\lambda,0)])](\omega,\ldots,\omega)\omega)\cdot e_i'}^2\\
&\leq nM^2\abs{[D_\eta^{m+1}a^1(\lambda,0)-D_\eta^{m+1}a^2(\lambda,0)])](\omega,\ldots,\omega)\omega}^2.\end{aligned}$$
Then, we have 
$$(1-nM^2)\abs{[D_\eta^{m+1}a^1(\lambda,0)-D_\eta^{m+1}a^2(\lambda,0)])](\omega,\ldots,\omega)\omega}^2\leq 0$$
and, using the fact that $1-nM^2>0$, we get
$$[D_\eta^{m+1}a^1(\lambda,0)-D_\eta^{m+1}a^2(\lambda,0)])](\omega,\ldots,\omega)\omega=0.$$
Fixing $i\in\{1,\ldots,p\}$, we obtain
$$[D_\eta^{m+1}a^1(\lambda,0)-D_\eta^{k+1}a^2(\lambda,0)](\omega,\ldots,\omega)\omega\cdot f_{i,1}^1(\lambda)=0.$$
Combining this with \eqref{t1h}, we obtain
$$D_\eta^{m+1}(\gamma_i^1-\gamma_i^2)(\lambda,0)(\omega,\ldots,\omega)(\omega\cdot f_{i,1}^1(\lambda))=0.$$
It follows that
\bel{t1l}D_\eta^{m+1}(\gamma_i^1-\gamma_i^2)(\lambda,0)(\omega,\ldots,\omega)=0\ee
holds true for  any $\omega\in\mathbb S^{n-1}$ satisfying $\omega\cdot f_{i,1}^1(\lambda)\neq0$.
Now let us fix $\omega'\in\mathbb S^{n-1}$ satisfying $\omega'\cdot f_{i,1}^1(\lambda)=0$. For any $q\in\N$, we have
$$D_\eta^{m+1}(\gamma_i^1-\gamma_i^2)(\lambda,0)((\sqrt{1-2^{-2q}})\omega'+2^{-q}f_{i,1}^1(\lambda),\ldots,(\sqrt{1-2^{-2q}})\omega'+2^{-q}f_{i,1}^1(\lambda))=0.$$
Sending $q\to+\infty$ and using the continuity of the map $\R^n\ni\xi\mapsto D_\eta^{m+1}(\gamma_i^1-\gamma_i^2)(\lambda,0)(\xi,\ldots,\xi)$, we obtain
$$D_\eta^{m+1}(\gamma_i^1-\gamma_i^2)(\lambda,0)(\omega',\ldots,\omega')=0$$
which proves that \eqref{t1l} holds true for any $\lambda\in[-R,R]$ and any $\omega\in\mathbb S^{n-1}$. Then, recalling that the map
$$(\xi_1,\ldots,\xi_{m+1})\mapsto D_\eta^{m+1}(\gamma_i^1-\gamma_i^2)(\lambda,0)(\xi_1,\ldots,\xi_{m+1})$$
is a symmetric $m+1$-linear map, by polarization (see e.g. \cite[Theorem 1]{T}), we deduce that
$$D_\eta^{m+1}\gamma_i^1(\lambda,0)=D_\eta^{m+1}\gamma_i^2(\lambda,0),\quad \lambda\in[-R,R].$$
Since $i\in\{1,\ldots,p\}$ is arbitrary chosen, we obtain \eqref{t1i} which implies \eqref{t1g}. This completes the proof of the theorem.
\subsection{Proof of Corollary \ref{c2}}
The proof of Corollary \ref{c2} follow the line of Theorem \ref{t1} with some modifications that will be mentioned here. Indeed, following the argumentation of Theorem \ref{t1} we can prove that \eqref{c2a}, for arbitrary chosen $\epsilon\in(0,\epsilon_\lambda)$ and $x_0\in\partial\Omega$, implies that
$$(\gamma^1(\lambda,0)-\gamma^2(\lambda,0))\nu(x_0)\cdot \omega=0,\quad \omega\in\mathbb S^{n-1}.$$
Choosing $\omega=\nu(x_0)$, we deduce that \eqref{t1c} holds true for $k=0$.

Now let us fix $m\in\{0,\ldots,\ell-2\}$ and let us assume that \eqref{t1f} is fulfilled. Fixing $\omega\in\mathbb S^{n-1}$ and repeating the argumentation of Theorem \ref{t1}, we deduce that
$$D_\eta^{m+1}(\gamma^1-\gamma^2)(\lambda,0)(\omega,\ldots,\omega)(\omega\cdot \nu(x_0))=0.$$
It follows that
$$D_\eta^{m+1}(\gamma^1-\gamma^2)(\lambda,0)(\omega,\ldots,\omega)=0$$
holds true for any $\omega\in\mathbb S^{n-1}$ satisfying $\omega\cdot \nu(x_0)\neq0$.
Now let us fix $\omega'\in\mathbb S^{n-1}$ satisfying $\omega'\cdot \nu(x_0)=0$. For any $q\in\N$, we have
$$D_\eta^{m+1}(\gamma_i^1-\gamma_i^2)(\lambda,0)((\sqrt{1-2^{-2q}})\omega'+2^{-q}\nu(x_0),\ldots,(\sqrt{1-2^{-2q}})\omega'+2^{-q}\nu(x_0))=0.$$
Sending $q\to+\infty$ and using the continuity of the map $\xi\mapsto D_\eta^{k+1}(\gamma^1-\gamma^2)(\lambda,0)(\xi,\ldots,\xi)$, we obtain
$$D_\eta^{m+1}(\gamma_i^1-\gamma_i^2)(\lambda,0)(\omega',\ldots,\omega')=0$$
which proves that \eqref{t1l} holds true for any $\omega\in\mathbb S^{n-1}$. Applying again the polarization argument we deduce that \eqref{t1g} holds true and we can conclude by iteration.

\section{Proof of the stability estimates}
\subsection{Proof of Theorem \ref{t2}}

We will prove the stability estimate \eqref{t2bb} as well as estimate \eqref{t2b}.  We start by considering \eqref{t2bb}. For this purpose, we fix $\lambda\in[-R,R]$, $\{e_1,\ldots,e_n\}$ an orthonormal basis of $\R^n$ and applying Lemma \ref{l1} and \ref{l2} we deduce that, for all $k=1,\ldots,n$, we have
$$\partial_\tau \partial_{\nu_{a^j}}u_{\lambda,e_k,\tau}^j(x)|_{\tau=0}=a^j(\lambda,0) e_k\cdot\nu(x)=b^j(\lambda) e_k\cdot\nu(x), j=1,2,\ x\in\partial\Omega.$$
Now fixing $M$ given by \eqref{M} and $k,j\in\{1,\ldots,n\}$, we get
$$ \begin{aligned}&\abs{((b^1(\lambda)-b^2(\lambda))e_k)\cdot e_j'}\\
&\leq \abs{\partial_\tau (\partial_{\nu_{a^1}}u_{\lambda,e_k,\tau}^1(x_j)-\partial_\tau \partial_{\nu_{a^2}}u_{\lambda,e_k,\tau}^2(x_j))|_{\tau=0}}+\abs{((b^1(\lambda)-b^2(\lambda))e_k)\cdot (e_j'-\nu(x_j)}\\
&\leq \sup_{\mu\in[-R,R]}\sup_{\theta\in\mathbb S^{n-1}}\abs{\partial_\tau (\partial_{\nu_{a^1}}u_{\mu,\theta,\tau}^1(x_j)-\partial_\tau \partial_{\nu_{a^2}}u_{\mu,\theta,\tau}^2(x_j))|_{\tau=0}}+M\abs{((b^1(\lambda)-b^2(\lambda))e_k)}.\end{aligned}$$
In view of \eqref{as}, we have $M<\frac{1}{\sqrt{n}}$ and we find
$$ \begin{aligned}&\abs{((b^1(\lambda)-b^2(\lambda))e_k)}^2\\
&=\sum_{j=1}^n\abs{((b^1(\lambda)-b^2(\lambda))e_k)\cdot e_j'}^2\\
&\leq n\left( \max_{j=1,\ldots,n}\sup_{\mu\in[-R,R]}\sup_{\theta\in\mathbb S^{n-1}}\abs{\partial_\tau (\partial_{\nu_{a^1}}u_{\mu,\theta,\tau}^1(x_j)-\partial_\tau \partial_{\nu_{a^2}}u_{\mu,\theta,\tau}^2(x_j))|_{\tau=0}}+M\abs{((b^1(\lambda)-b^2(\lambda))e_k)}\right)^2\\
&\leq (n+\frac{2n}{n^{-1}-M^2})\left(\max_{j=1,\ldots,n}\sup_{\mu\in[-R,R]}\sup_{\theta\in\mathbb S^{n-1}}\abs{\partial_\tau (\partial_{\nu_{a^1}}u_{\mu,\theta,\tau}^1(x_j)- \partial_{\nu_{a^2}}u_{\mu,\theta,\tau}^2(x_j))|_{\tau=0}}\right)^2\\
&\ \ \ +\frac{1+nM^2}{2}\abs{((b^1(\lambda)-b^2(\lambda))e_k)}^2.\end{aligned}$$
It follows
$$\begin{aligned}&\abs{((b^1(\lambda)-b^2(\lambda))e_k)}^2\\
&\leq\frac{2(n+\frac{2n}{n^{-1}-M^2})}{1-nM^2}\left(\max_{j=1,\ldots,n}\sup_{\mu\in[-R,R]}\sup_{\theta\in\mathbb S^{n-1}}\abs{\partial_\tau (\partial_{\nu_{a^1}}u_{\mu,\theta,\tau}^1(x_j)- \partial_{\nu_{a^2}}u_{\mu,\theta,\tau}^2(x_j))|_{\tau=0}}\right)^2\end{aligned}$$
which implies 
$$\begin{aligned}&\abs{b^1(\lambda)-b^2(\lambda)}_{\R^{n\times n}}\\
&\leq  \sqrt{\frac{2(n+\frac{2n}{n^{-1}-M^2})}{1-nM^2}}\max_{j=1,\ldots,n}\sup_{\mu\in[-R,R]}\sup_{\theta\in\mathbb S^{n-1}}\abs{\partial_\tau (\partial_{\nu_{a^1}}u_{\mu,\theta,\tau}^1(x_j)- \partial_{\nu_{a^2}}u_{\mu,\theta,\tau}^2(x_j))|_{\tau=0}}.\end{aligned}$$
From this estimate, we deduce \eqref{t2bb}.

Now let us show \eqref{t2b}. We start by considering \eqref{t2b} for $N=1$.  In view of Lemma \ref{l1} and \ref{l2}, for any $\lambda\in[-R,R]$ and $\omega\in\mathbb S^{n-1}$, we have
$$\begin{aligned}&\partial_\tau^2 \partial_{\nu_{a^j}}u_{\lambda,\omega,\tau}^j(x)|_{\tau=0}\\
&=(x\cdot\omega)\partial_\mu b^j(\lambda)\omega\cdot\nu(x)+D_\eta\gamma_j(\lambda,0)(\omega) b^j(\lambda)\omega\cdot\nu(x)+b^j(\lambda)\nabla y^j_\lambda\cdot\nu(x), j=1,2,\ x\in\partial\Omega,\end{aligned}$$
where $y^j_\lambda$ solves the problem
$$
\left\{
\begin{array}{ll}
-\sum_{k,\ell=1}^n 
 b_{k,\ell}^j(\lambda) \partial_{x_k} \partial_{x_\ell} y^j_\lambda =\partial_\mu b^j(\lambda)\omega\cdot\omega  & \mbox{in}\ \Omega ,
\\
y^j_\lambda=0 &\mbox{on}\ \partial\Omega.
\end{array}
\right.$$

Fixing $b=b^1-b^2$, $\gamma=\gamma_1-\gamma_2$ and $h_\tau=\partial_{\nu_{a^1}}u_{\lambda,\omega,\tau}^1-\partial_{\nu_{a^2}}u_{\lambda,\omega,\tau}^2$,  for any $\lambda\in[-R,R]$, $\omega\in\mathbb S^{n-1}$ and $k\in\{1,\ldots,n\}$, we obtain
\bel{t2d}\begin{aligned}& D_\eta\gamma(\lambda,0)(\omega) b^1(\lambda)\omega\cdot\nu(x_k)\\
&=\partial_\tau^2 h_\tau(x_k)|_{\tau=0}-(x_k\cdot\omega)\partial_\mu b(\lambda)\omega\cdot\nu(x_k)-D_\eta\gamma^2(\lambda,0)(\omega)b(\lambda)\omega\cdot\nu(x_k)\\
&\ \ -b(\lambda)\nabla y^1_\lambda\cdot\nu(x_k)-b^2(\lambda)\nabla y_\lambda\cdot\nu(x_k)\\
&=\partial_\tau^2 h_\tau(x_k)|_{\tau=0}+I+II+III+IV,\end{aligned}\ee
where $y_\lambda=y^1_\lambda-y^2_\lambda$ solves the problem 
$$
\left\{
\begin{array}{ll}
-\sum_{k,\ell=1}^n 
 b_{k,\ell}^1(\lambda) \partial_{x_k} \partial_{x_\ell} y_\lambda =\nabla\cdot (b(\lambda)\nabla y^2_\lambda)+\partial_\mu b(\lambda)\omega\cdot\omega  & \mbox{in}\ \Omega ,
\\
y_\lambda=0 &\mbox{on}\ \partial\Omega.
\end{array}
\right.$$      
Using the formula \eqref{t2d}, we will start by proving that
\bel{t2e}\begin{aligned}&| D_\eta\gamma(\lambda,0)(\omega) b^1(\lambda)\omega\cdot\nu(x_k)| \\
&\leq C\sum_{k=1}^{2}\left(\max_{j=1,\ldots,n}\sup_{\theta\in \mathbb S^{n-1}}\sup_{\mu\in[-R,R]}\abs{\partial_\tau^k(\partial_{\nu_{a^1}} u_{\mu,\theta,\tau}^1(x_j)-\partial_{\nu_{a^2}} u_{\mu,\theta,\tau}^2(x_j))|_{\tau=0}}\right)^{\frac{3^k}{3^{2}}}.\end{aligned}\ee
For this purpose, we will estimate the different expression $I$, $II$, $III$, $IV$ appearing in \eqref{t2d}.

For $I$, we have
$$\abs{(x_k\cdot\omega)\partial_\mu b(\lambda)\omega\cdot\nu(x_k)}\leq C\norm{\partial_\mu b}_{L^\infty((-R,R);\R^{n\times n})},\quad \lambda\in[-R,R],$$
with $C>0$ depending only on $\Omega$. By interpolation (see e.g. \cite[Lemma AppendixB.1.]{CK}), we obtain
$$\norm{\partial_\mu b}_{L^\infty((-R,R);\R^{n\times n})}\leq C\norm{ b}_{W^{3,\infty}((-R,R);\R^{n\times n}))}^{\frac{1}{3}}\norm{\partial_\mu b}_{L^2((-R,R);\R^{n\times n})}^{\frac{2}{3}},$$
$$\norm{\partial_\mu b}_{L^2((-R,R);\R^{n\times n})}\leq C\norm{ b}_{L^2((-R,R);\R^{n\times n})}^{\frac{1}{2}}\norm{ b}_{H^2((-R,R);\R^{n\times n})}^{\frac{1}{2}},$$
with $C>0$ depending only on $R$. Combining this with \eqref{t2a} and \eqref{t2bb}, we find
\bel{t2f}\norm{\partial_\mu b}_{L^\infty((-R,R);\R^{n\times n})}\leq C\left(\max_{j=1,\ldots,n}\sup_{\theta\in \mathbb S^{n-1}}\sup_{\mu\in[-R,R]}\abs{\partial_\tau(\partial_{\nu_{a^1}} u_{\mu,\theta,\tau}^1(x_j)-\partial_{\nu_{a^2}} u_{\mu,\theta,\tau}^2(x_j))|_{\tau=0}}\right)^{\frac{1}{3}},\ee
with $C>0$ depending on $R$, $\psi$, $x_1,\ldots,x_n$ and $\Omega$.
Therefore, we have
\bel{t2g}|I|\leq C\left(\max_{j=1,\ldots,n}\sup_{\theta\in \mathbb S^{n-1}}\sup_{\mu\in[-R,R]}\abs{\partial_\tau(\partial_{\nu_{a^1}} u_{\mu,\theta,\tau}^1(x_j)-\partial_{\nu_{a^2}} u_{\mu,\theta,\tau}^2(x_j))|_{\tau=0}}\right)^{\frac{1}{3}}.\ee

In the same way, combining \eqref{t2a} and \eqref{t2bb}, we deduce that 
\bel{t2h}|II|\leq C\left(\max_{j=1,\ldots,n}\sup_{\theta\in \mathbb S^{n-1}}\sup_{\mu\in[-R,R]}\abs{\partial_\tau(\partial_{\nu_{a^1}} u_{\mu,\theta,\tau}^1(x_j)-\partial_{\nu_{a^2}} u_{\mu,\theta,\tau}^2(x_j))|_{\tau=0}}\right)^{\frac{1}{3}},\ee
with $C>0$ depending on $R$, $\psi$, $x_1,\ldots,x_n$ and $\Omega$.

For $III$, note first that the function $\kappa$ in \eqref{ell2} satisfies the condition
$$c=\inf_{\mu\in[-R,R]}\kappa(\mu,0)>0$$
and we have
\bel{t2ab} a^j(\mu,0)\xi\cdot\xi\geq c|\xi|^2,\quad \mu\in[-R,R],\ j=1,2.\ee
In addition, condition \eqref{t2a} implies
\bel{t2ba} |a^j(\mu,0)|_{\R^{n\times n}}\leq \psi(R),\quad \mu\in[-R,R],\ j=1,2.\ee
Combining \eqref{t2ab}-\eqref{t2ba} with \cite[Theorem 6.6]{GT}, we deduce that there exists $C_3>0$ depending on $\kappa$, $\psi$, $\Omega$ and $R$ such that 
\bel{t2i}\norm{y^j_\lambda}_{C^{2+\alpha}(\overline{\Omega})}\leq C_3,\quad j=1,2,\ \lambda\in[-R,R].\ee
It follows that
$$\abs{III}\leq C\abs{b_1(\lambda)-b_2(\lambda)},\quad \lambda\in[-R,R]$$
and, combining this with \eqref{t2bb}, we get
\bel{t2j}|III|\leq C\left(\max_{j=1,\ldots,n}\sup_{\theta\in \mathbb S^{n-1}}\sup_{\mu\in[-R,R]}\abs{\partial_\tau(\partial_{\nu_{a^1}} u_{\mu,\theta,\tau}^1(x_j)-\partial_{\nu_{a^2}} u_{\mu,\theta,\tau}^2(x_j))|_{\tau=0}}\right)^{\frac{1}{3}},\ee

Finally, for $IV$, applying  \eqref{t2ab}-\eqref{t2ba}, \eqref{t2i} and \cite[Theorem 6.6]{GT}, we obtain
$$\norm{y_\lambda}_{C^{2+\alpha}(\overline{\Omega})}\leq C\left(\norm{ b}_{L^\infty((-R,R);\R^{n\times n})}+\norm{\partial_\mu b}_{L^\infty((-R,R);\R^{n\times n})}\right),\quad \lambda\in[-R,R],$$
with $C>0$  depending on $\kappa$, $\psi$, $\Omega$ and $R$. Then, \eqref{t2a} and \eqref{t2f} imply that, for all $\lambda\in[-R,R]$, we have
$$\norm{y_\lambda}_{C^{2+\alpha}(\overline{\Omega})}\leq C\left(\max_{j=1,\ldots,n}\sup_{\theta\in \mathbb S^{n-1}}\sup_{\mu\in[-R,R]}\abs{\partial_\tau(\partial_{\nu_{a^1}} u_{\mu,\theta,\tau}^1(x_j)-\partial_{\nu_{a^2}} u_{\mu,\theta,\tau}^2(x_j))|_{\tau=0}}\right)^{\frac{1}{3}}.$$
From this last estimate, we deduce that
$$|IV|\leq C\left(\max_{j=1,\ldots,n}\sup_{\theta\in \mathbb S^{n-1}}\sup_{\mu\in[-R,R]}\abs{\partial_\tau(\partial_{\nu_{a^1}} u_{\mu,\theta,\tau}^1(x_j)-\partial_{\nu_{a^2}} u_{\mu,\theta,\tau}^2(x_j))|_{\tau=0}}\right)^{\frac{1}{3}}.$$
Combining this estimate with \eqref{t2d}, \eqref{t2g}, \eqref{t2h}, \eqref{t2j}, we obtain \eqref{t2e}. 

We will now show that \eqref{t2e} implies \eqref{t2b} for $N=1$. For this purpose, we set
$$s:=\sum_{k=1}^{2}\left(\max_{j=1,\ldots,n}\sup_{\theta\in \mathbb S^{n-1}}\sup_{\mu\in[-R,R]}\abs{\partial_\tau^k(\partial_{\nu_{a^1}} u_{\mu,\theta,\tau}^1(x_j)-\partial_{\nu_{a^2}} u_{\mu,\theta,\tau}^2(x_j))|_{\tau=0}}\right)^{\frac{3^k}{3^{2}}},$$
$\lambda\in[-R,R]$ and $\omega\in\mathbb S^{n-1}$.
Fixing $M$ given by \eqref{M} and $k\in\{1,\ldots,n\}$, we get
$$ \begin{aligned}&\abs{ D_\eta\gamma(\lambda,0)(\omega) b^1(\lambda)\omega\cdot e_k'}\\
&\leq Cs+\abs{(D_\eta\gamma(\lambda,0)(\omega) b^1(\lambda)\omega)\cdot (e_k'-\nu(x_k)}\\
&\leq Cs+M\abs{D_\eta\gamma(\lambda,0)(\omega) b^1(\lambda)\omega}.\end{aligned}$$
In light of \eqref{as}, we have $M<\frac{1}{\sqrt{n}}$ and we find
$$ \begin{aligned}&\abs{D_\eta\gamma(\lambda,0)(\omega) b^1(\lambda)\omega}^2\\
&=\sum_{j=1}^n\abs{(D_\eta\gamma(\lambda,0)(\omega) b^1(\lambda)\omega)\cdot e_k'}^2\\
&\leq n\left( Cs+M\abs{D_\eta\gamma(\lambda,0)(\omega) b^1(\lambda)\omega}\right)^2\\
&\leq (n+\frac{2n}{n^{-1}-M^2})Cs^2 +\frac{1+nM^2}{2}\abs{D_\eta\gamma(\lambda,0)(\omega) b^1(\lambda)\omega}^2.\end{aligned}$$
It follows
$$\abs{D_\eta\gamma(\lambda,0)(\omega)}\abs{ b^1(\lambda)\omega}_{\R^{n}}\leq  C\sum_{k=1}^{2}\left(\max_{j=1,\ldots,n}\sup_{\theta\in \mathbb S^{n-1}}\sup_{\mu\in[-R,R]}\abs{\partial_\tau^k(\partial_{\nu_{a^1}} u_{\mu,\theta,\tau}^1(x_j)-\partial_{\nu_{a^2}} u_{\mu,\theta,\tau}^2(x_j))|_{\tau=0}}\right)^{\frac{3^k}{3^{2}}}.$$
On the other hand, in view of \eqref{ell1}-\eqref{ell2} and \eqref{t2aa}, one can check that there exists $c>0$ depending on $\kappa$ and $R$, such that
$$\abs{ b^1(\lambda)\omega}_{\R^{n}}\geq c,\quad \lambda\in[-R,R],\ \omega\in\mathbb S^{n-1}.$$
Therefore, we have
$$\begin{aligned}&\sup_{\mu\in[-R,R]}\sup_{\theta\in\mathbb S^{n-1}}\abs{D_\eta\gamma(\mu,0)(\theta)}\\
&\leq  C\sum_{k=1}^{2}\left(\max_{j=1,\ldots,n}\sup_{\theta\in \mathbb S^{n-1}}\sup_{\mu\in[-R,R]}\abs{\partial_\tau^k(\partial_{\nu_{a^1}} u_{\mu,\theta,\tau}^1(x_j)-\partial_{\nu_{a^2}} u_{\mu,\theta,\tau}^2(x_j))|_{\tau=0}}\right)^{\frac{3^k}{3^{2}}}\end{aligned}$$
which implies \eqref{t2b} for $N=1$.

In a similar way to Theorem \ref{t1}, we can show by iteration of the above argumentation that, for any $N=1,\ldots,\ell-1$, we have
$$ \begin{aligned}&\sup_{\mu\in[-R,R]}\sup_{\theta\in\mathbb S^{n-1}}\abs{D_\eta^N\gamma(\mu,0)(\theta,\ldots,\theta)}\\
&\leq C\sum_{k=1}^{N+1}\left(\max_{j=1,\ldots,n}\sup_{\theta\in \mathbb S^{n-1}}\sup_{\mu\in[-R,R]}\abs{\partial_\tau^k(\partial_{\nu_{a^1}} u_{\mu,\theta,\tau}^1(x_j)-\partial_{\nu_{a^2}} u_{\mu,\theta,\tau}^2(x_j))|_{\tau=0}}\right)^{\frac{3^k}{3^{N+1}}},\end{aligned}$$
from which we can derive \eqref{t1b} by polarization (see e.g. \cite[Theorem 1]{T}). This completes the proof of the theorem.

\subsection{Proof of Corollary \ref{c3}}

The proof of Corollary \ref{c3} follows the line of Theorem \ref{t2} with some modifications that will be mentioned here. Indeed, following the argumentation of Theorem \ref{t2} we can show that, for any $\omega\in\mathbb S^{n-1}$, we have
$$\begin{aligned}\abs{(\gamma^1(\lambda,0)-\gamma^2(\lambda,0))\nu(x_0)\cdot \omega}\leq C_1\left(\sup_{\theta\in\mathbb S^{n-1}}\abs{\partial_\tau (\partial_{\nu_{a^1}}u_{\lambda,\theta,\tau}^1(x_0)- \partial_{\nu_{a^2}}u_{\lambda,\theta,\tau}^2(x_0))|_{\tau=0}}\right)\end{aligned}$$
with $C_1$ depending only on $n$ and $x_0$. Choosing $\omega=\nu(x_0)$, we obtain \eqref{c3a}.
In the same way, repeating the iteration process described in the proof of Theorem \ref{t2}, we obtain 
$$ \begin{aligned}&\sup_{\mu\in[-R,R]}\sup_{\theta\in\mathbb S^{n-1}}\abs{D_\eta^N\gamma(\mu,0)(\theta,\ldots,\theta)}\\
&\leq C\sum_{k=1}^{N+1}\left(\sup_{\theta\in \mathbb S^{n-1}}\sup_{\mu\in[-R,R]}\abs{\partial_\tau^k(\partial_{\nu_{a^1}} u_{\mu,\theta,\tau}^1(x_0)-\partial_{\nu_{a^2}} u_{\mu,\theta,\tau}^2(x_0))|_{\tau=0}}\right)^{\frac{3^k}{3^{N+1}}},\end{aligned}$$
and we find \eqref{c3b} by polarization. This completes the proof of the corollary.

\bigskip
\vskip 1cm

\end{document}